\documentclass[10pt]{amsart}
\usepackage{latexsym, amsmath,amssymb, bbold}

	\usepackage{xcolor}

\renewcommand{\epsilon}{\varepsilon}
\newcommand{\newsection}[1]
{\subsection{#1}\setcounter{theorem}{0} \setcounter{equation}{0}
\par\noindent}

\newtheorem{theorem}{Theorem}

\newtheorem{lemma}[theorem]{Lemma}
\newtheorem{corr}[theorem]{Corollary}

\newtheorem{proposition}[theorem]{Proposition}
\newtheorem{deff}[theorem]{Definition}

\newtheorem{exmp}[theorem]{Example}
\newcommand{\bth}{\begin{theorem}}
\newcommand{\ble}{\begin{lemma}}
\newcommand{\bcor}{\begin{corr}}

\newcommand{\bdeff}{\begin{deff}}

\newcommand{\bprop}{\begin{proposition}}
\newcommand{\ele}{\end{lemma}}
\newcommand{\ecor}{\end{corr}}
\newcommand{\edeff}{\end{deff}}

\newcommand{\eprop}{\end{proposition}}

\newcommand{\supp}{\text{supp }}
\renewcommand{\Pi}{\varPi}

\renewcommand{\epsilon}{\varepsilon}

\newcommand{\R}{{\mathbb R}}

\newcommand{\vol}{\operatorname{vol}}
\newcommand{\Z}{{\mathbb Z}}
\newcommand{\N}{{\mathbb N}}
\newcommand{\T}{{\mathbb T}}

\makeatletter
\newcommand*{\rom}[1]{\expandafter\@slowromancap\romannumeral #1@}
\makeatother

\begin{document}

\title[Improved Estimates on Product Manifolds]
{Weyl formula improvement for product of zoll manifolds}
%
%
%

\author[]{Yanfei Wang}
\address[Y.F.W.]{Department of Mathematics,  Johns Hopkins University,
Baltimore, MD 21218}
\email{ywang513@jh.edu}

\begin{abstract}
Iosevich and Wyman have proved in ~\cite{IoWy} that the remainder term in classical Weyl law can be improved from $O(\lambda^{d-1})$ to $o(\lambda^{d-1})$ in the case of product manifold by using a famous result of Duistermaat and Guillemin. They also showed that we could have polynomial improvement in the special case of Cartesian product of round spheres by reducing the problem to the study of the distribution of weighted integer lattice points. In this paper, we show that we can extend this result to the case of Cartesian product of Zoll manifolds by investigating the eigenvalue clusters of Zoll manifold and reducing the problem to the study of the distribution of weighted integer lattice points too.
\end{abstract}

\maketitle


\newsection{Introduction}


Let M be a compact Riemannian manifold without boundary. $\Delta_M$ denote the Laplace-Beltrami operator of this manifold. We can use spectral theorem to decompose 
\[
    L^2(M) =  \bigoplus_{\lambda \in \Lambda} E_\lambda ,
\]
where
\[
    \Lambda = \{\lambda \in [0,\infty) : -\lambda^2 \text{ is an eigenvalue of $\Delta$}\}
\]
is the spectrum of $\sqrt{-\Delta_M}$ and $E_{\lambda}$ is the eigenspace corresponding to the eigenvalue $\lambda$. The dimension of $E_{\lambda}$ is called the multiplicity of $\lambda$ and will be denoted $\mu(\lambda)$. The Weyl counting function is usually defined by:
\[
N(\lambda) = \sum_{\lambda'\in \Lambda \cap [0,\lambda]} \mu(\lambda')
\]
and we have the integrated Weyl law,
\begin{equation}\label{the weyl law}
    N(\lambda) = \frac{|B_d| \vol M}{(2\pi)^d} \lambda^{d} + O(\lambda^{d - 1})
\end{equation}
where d is the dimension of M , $\vol M$ is the volume of M and $|B_d|$ is the volume of the unit ball in $R^d$.

We also have the piontwise Weyl law uniformly for $x\in M$, which can deduce the intergrated version above:
\[
\textbf{1}(- \Delta_g \le \lambda^2)(x,x) = \frac{|B_d|}{(2\pi)^d} \lambda^{d} + O(\lambda^{d - 1}) 
\]
On closed manifold, the sharp pointwise Weyl law is due to Avakumovic ~\cite{Avakumovic}, Levitan~\cite{Levitan}, and Hormander~\cite{Hor68}. Here, sharp means that we can find manifolds such that the remainder $O(\lambda^{d-1})$, can not be improved, for example, on spheres and Zoll manifolds. We will briefly discuss the spectrum of Laplacian on round spheres and Zoll manifolds in the following examples.

\begin{exmp}(Round Sphere) The spectrum of the Laplace-Beltrami operator on $S^d$ is given by 
\[
    \Lambda = \left\{\sqrt{k(k + d - 1)} : k \in \Z_{\geq 0} \right\}
\]
with corresponding multiplicities
\begin{equation}\label{sphere spectrum}
    \mu(\sqrt{k(k + d - 1)}) = \binom{d+k}{d} - \binom{d+k-2}{d} = \frac{2}{(d-1)!} k^{d-1} + O(k^{d-2}).
\end{equation}
This means that the Weyl conunting function will have jumps of order $O(\lambda^{d-1})$ which saturates the standard remainder term. (See Sogge ~\cite{soggefica} section 3.4.)
\end{exmp}

The big-O term in the Weyl law can be improved under some additional geometric assumptions for the manifold. For example, from the work of Duistermaat and Guillemin ~\cite{DuistermaatGuillemin} we can get a qualitative improvement if the set of closed geodesics is small enough. We will have a $\log \lambda$ improvement if the sectional curvature is nonpositive, see Berard ~\cite{Berard}. From a recent work of Canzani and Galkowski ~\cite{CanGal}, we can obtain a $\log \lambda$ improvement in the case of product manifolds. Taylor, Sogge and Huang proved that improved Weyl remainder terms bound on $Y$ could lead to a imporved Weyl remainder terms bound on $X\times Y$ in \cite{HuangTaylorSogge}. From a recent work of Wyman and Iosevich~\cite{IoWy}, we have a polynomial improvement if we consider product of spheres. They also raise a conjecture that we can obtain polynomial improvements for product of general manifolds.

The main purpose of this papaer is to extend the rusults of Wyman and Iosevich for product of spheres to the product of Zoll manifolds. Before stating our theorem, we want to briefly introduce some old results about Zoll manifold which will be useful in our proof.

From here, the manifold Z is a so-called Zoll manifold, a compact Riemannian manifold without boundary all of whose geodesics are closed with the same common period T. To normalize it, we let $T=2\pi$. (See ~\cite{Zollmanifold} for a thorough study of such manifolds.) The spectrum of Zoll manifolds is similar to the spectrum of round spheres, except for a major difference. Sphere spectrum will accumulate to a series of points while Zoll spectrum is just clustered in a series of intervals whose length will shrink in the speed of $O(1/k)$ (see Duistermaat and Guillemin~\cite{DuistermaatGuillemin} and Weinstein ~\cite{Wein77} for more details). The number of eigenvalues inside these intervals will also be similar to the sphere case. For the sphere case, Iosevich and Wyman~\cite{IoWy} has shown us we can deduce from \eqref{sphere spectrum} that 
\begin{equation*}
    P_{S^d}\left(k + \frac{d-1}{2}\right) = \frac{2}{(d-1)!}\left(k + \frac{d-1}{2}\right)\frac{(k+d-2)!}{k!} \qquad \text{ if $k \geq 2$.}
\end{equation*}
where $ P_{S^d}(k + (d-1)/2)$ is the number of eigenvalues inside the eigenspace $E_k$. We want to analyze the order of this polynomial, i.e. the asymptotic behavior of the eigenvalue distribution. Here, if you consider P as a polynomial of k, then \eqref{sphere spectrum} is the best asymptotics you can get. However, if you change your variable to $t=k+\frac{d-1}{2}$, you will get:
\begin{equation} \label{polynomial}
    P_{S^d}(t) = \frac{2}{(d-1)!} t^{d-1} + O(t^{d-3}) \qquad \text{ for $|t|$ large.}
\end{equation}
This fact, at our first glimpse, may rely on the special property of combinatorial number. However, \eqref{sphere spectrum} can be computed use combinatorial numbers is just because the round sphere have very good symmetry in all directions. (See ~\cite{SoggeHangzhou} for details.) Although Zoll manifold does not have this geometric symmetry, we can prove some similar result as \eqref{polynomial} on Zoll manifold. 

\begin{exmp}\label{Zoll Manifold}(Zoll Manifold)
Let Z be a d dimensional Zoll manifold all of whose geodesics have the same period $2\pi$. The spectrum of its Laplace-Beltrami operator $-\Delta$ lies in clusters about the squares $(k + \frac{\alpha}{4})^2$. Where $\alpha$ is the Maslov index of this Zoll manifold. For all eigenvalues $\sqrt{\nu_{n(k,j)}}$ of $\sqrt{-\Delta_Z}$ inside the k'th clusters, we have the following estimates,
\begin{equation}\label{Zollcluster}
     \sqrt{\nu_{n(k,j)}} =  k + \frac{\alpha}{4 } + O(k^{-1})
\end{equation}
See Zelditch ~\cite{FSZS} for more details.
\end{exmp}

\begin{lemma}\label{zec}
    Z is still our Zoll manifold defined as above, the number of eigenvalues in the k'th cluster as a polynomial of $k+\frac{\alpha}{4}$ is 
    \begin{equation} \label{zolclu}
    P_{Z^d}(t) = C \cdot t^{d-1} + O(t^{d-3}) \qquad \text{ for $|t|$ large.}
    \end{equation}
    where $t=k + \frac{\alpha}{4}$ and C is a constant depend on Z.
\end{lemma}

With the aboving lemma, we can prove the following theorem, which tells us that we can have polynomial improvements for the remainder term of Weyl law if our manifold is a Cartesian product of zoll manifold.

\begin{theorem}\label{zollweyl}
    Let $M = Z^{d_1} \times \cdots \times Z^{d_n}$ be a product of $n$ Zoll manifolds of dimension $d_i \geq 1$ for $i = 1,\ldots,n$. Let $\Delta$ be the Laplace-Beltrami operator with respect to the product metric on $M$. If $-\lambda_j^2$ for $j = 1,2,\ldots$ are the eigenvalues of $\Delta$ repeated with multiplicity, then we have bounds
    \[
        N(\lambda) := \#\{ j : \lambda_j \leq \lambda \} = \frac{|B_{|d|}|}{(2\pi)^{|d|}} \vol(M) \lambda^{|d|} + O(\lambda^{|d|-1-\frac{n-1}{n+1}}).
    \]
    where $|d|$ denotes $\dim M = d_1 + \cdots + d_n$.
\end{theorem}

\newsection{Proof of Lemma \ref{zec}}

Our proof relies on the work of Weinstein ~\cite{Wein77}, Duistermatt and Guillemin ~\cite{DuistermaatGuillemin}. For the completeness, we will include some theorems of those two papers in this chapter.

First, let $\Delta$ to be the Laplace Beltrami operator for our Zoll manifold Z, we can define an positive elliptic self-adjoint pseudo differential operator $A_0$ such that 
\[
A_0^2 = -\Delta + cI
\]
for some constant c. 

The main idea is to decompose the Laplacian to a square of first order "periodic operator" $A^2$ plus a harmless 0 order remainder term $B$, $-\Delta = A^2 + B$, and use averaging method to modify $B$ such that it can commute with A. Our $A_0$ is not periodic but the bicharacteristic flow of it is simply periodic with the period $2\pi$ since Z is a Zoll manifold all of whose geodesics are closed with the same length $2\pi$. This gives us 
\[
e^{2\pi i A_0}=P
\]
P is a unitary pseudo-differential operator of order 0. Since the subprincipal symbol of $A_0$ is zero (by the equation $A_0^2 = -\Delta +cI$) , the principal symbol of P is simply the constant $e^{\pi i \alpha/2 }$ , where $\alpha$ is the Maslov index of our manifold Z. (See Duistermaat and Guillemin~\cite{DuistermaatGuillemin} , Duistermaat and Hormander~\cite{FIO2} for details, and we will briefly introduce subprincipal symbol and prove $sub(A_0) = 0 $ later in this chapter).

Now we have 
\begin{equation*}
    e^{2\pi i (A_0-\frac{\alpha}{4})} = I + Q
\end{equation*}
where Q is a pseudo differential operator of order -1. The unitray operator $I+Q$ has discrete spectrum approaching 1, so we can choose a single valued branch $\mathbb{L}$ of the complex logarithm which is purely imaginary on this spectrum and such that $\mathbb{L}(1)= 0$. By functional calculus of pseudo differential operator (see Seeley ~\cite{Seeley1967ComplexPO}), we can define:
\begin{equation*}
    T = \frac{1}{2\pi i} \mathbb{L} (I+Q)
\end{equation*}
Where T is  a self-adjoint pseudo differential operator of order 0. And it is not hard to see the principal symbol of T is 0 because $\mathbb{L}(1) = 0$. So T has order -1. Also by definition, T will commute with $A_0$, and we have:
\begin{equation*}
    e^{2\pi i (A_0-\frac{\alpha}{4}-T)} = e^{2\pi i (A_0-\frac{\alpha}{4})} e^{-2\pi i T} = I
\end{equation*}
Now we can define $A = A_0 - T$. A is periodic now since $e^{2\pi i A} = e^{\pi i \alpha /2} I$. A is also elliptic since $A_0 = \sqrt{-\Delta + c}$ and T is of order -1. And we can make our decomposition $-\Delta = A^2 + B$, where B is :
\begin{align*}
    B &= -\Delta - A^2\\
      &= (A_0^2 - cI) - (A_0 -T)^2\\
      &= A_0^2 - cI - A_0^2 + 2TA_0 - T^2\\
      &= -cI + 2TA_0 - T^2
\end{align*}
Here T is self-adjoint of order -1 and $A_0$ is self-adjoint of order 1, so B is self-adjoint of order 0.

The following lemma is to show that we can modify B to a new operator $\Bar{B}$ by averaging, such that $\Bar{B}$ commute with A.
\begin{lemma}\label{com}
    Let A be a first order, self-adjoint, positive ellliptic pseudo differential operator, B is self-adjoint and of order 0. First we define $B_t = e^{-itA} B e^{itA}$, and then define the averaged operator $\Bar{B} = \frac{1}{2\pi}\int_0^{2\pi}B_t dt$. A commute with $\Bar{B}$,  $[A,\Bar{B}] = 0$
\end{lemma}

\begin{proof}
    First, $e^{itA} $ is a unitary Fourier integral operator, so the $B_t$ and $\Bar{B}$ are both self-adjoint of order 0. See ~\cite{DuistermaatGuillemin} for more details.
    
    We can take derivatives of $B_t$:
    \begin{equation}\label{commu}
        \frac{d}{dt} B_t = -iAe^{-itA}B e^{itA} + ie^{-itA}Be^{-itA}A = \frac{1}{i}[A,B_t]
    \end{equation}
    Now we have 
    \begin{equation*}
        [A,\Bar{B}] = \frac{i}{2\pi}\int_0^{2\pi} \frac{d}{dt} dt = \frac{i}{2\pi}(B_{2\pi} - B_0)
    \end{equation*}
    Since $e^{2\pi i A} = cI$  commute with B, we know that $B_{2\pi} - B_0 = 0$, which gives us $[A,\Bar{B}] = 0$ and finish our proof.
\end{proof}

By definition of A we have $e^{2\pi i A} = cI$. Let $\lambda$ be one of the eigenvalues of A, then $\lambda$ will satisfy $e^{2\pi i \lambda} = c$. So the spectrum of A will be:
\[
    \Lambda_A = \left\{k+\lambda_0 : k \in \Z_{\geq 0} \right\}
\]
where $\lambda_0$ is the lowest eigenvalue, which depends on c. Denote the dimension of eigenspace $E_k = ker(A - (k+ \lambda_0)I)$ by $d_k$. 

Since A commute with $\Bar{B}$, $E_k$ is invariant under $\Bar{B}$, we can rearrange the basis of $E_k$ such that they are also the basis of $\Bar{B}$. Then we just simultaneously diagonalize A and $\Bar{B}$, so we can just set the spectrum of $A^2 + \Bar{B}$ as:
\begin{equation*}
     \Lambda_0 = \left\{\lambda_j^k = (k+\lambda_0)^2 + \mu_j^k: \quad k \geq 0 \quad \textrm{and} \quad 1 \leq j \leq d_k \right\}
\end{equation*}
We know that B is of order 0, so it is a bounded operator, we can rearrange the numbers $\mu_j^k$ such that $-||\Bar{B}|| \leq \mu_1^k \leq \dots \leq u_{d_{k}}^k \leq ||\Bar{B}||$. Then the numbers $\lambda_j^k$ are generally in a increasingly order like $\lambda_1^0$ , $\dots$, $\lambda_{d_0}^0$ , $\lambda_1^1$ , $\dots$, $\lambda_{d_1}^1$ , $\dots$ except for finitely many cases. However, since we just care about asymptotic behavior in the following of this paper, finitely many cases don't hurt us too much. Now we can define the spectrum of $A^2 + B$ as the following:
\begin{equation*}
     \Lambda_1 = \left\{\nu_{n(j,k)} = (k+\lambda_0)^2 + \nu_j^k: \quad k \geq 0 \quad \textrm{and} \quad 1 \leq j \leq d_k \right\}
\end{equation*}
where our $n(j,k) = d_1 + d_2 + \dots + d_{k-1} + j$. Now we claim that the difference between $\nu_{n(j,k)}$ and $\lambda_j^k$ won't be greater than $O(1/k)$. 
\begin{proposition}\label{difference}
    Let $\nu_{n(j,k)}$ and $\lambda_j^k$ be the eigenvalues of operator $A^2+B$ and $A^2+\Bar{B}$. Then:
    \begin{equation}
        |\nu_{n(j,k)} - \lambda_j^k | \leq \frac{M}{k}
    \end{equation}
    for all j and k.
\end{proposition}

\begin{proof}
    We first make a claim which will be proved later in another lemma that there exists a unitary operator $U$ such that the difference $R = U ( A^2 +B) U^{-1} - (A^2 + \Bar{B})$ is a self adjoint operator of order -1. So $A^{1/2}RA^{1/2}$ is a bounded operator whose norm is denoted by m. Then we have
    \[
     - mI \leq A^{1/2}RA^{1/2} \leq mI
    \]
    So
    \[
     - mA^{-1} \leq R \leq mA^{-1}
    \]
    and 
    \[
    A^2 + \Bar{B} - mA^{-1} \leq A^2 + \Bar{B} +R \leq A^2 + \Bar{B} + mA^{-1}
    \]
    which is 
    \[
    A^2 + \Bar{B}- mA^{-1} \leq  U(A^2 +B )U^{-1}\leq A^2 + \Bar{B} + mA^{-1}
    \]
    The $n(j,k)$'th eigenvalue of $A^2 +\Bar{B} \pm mA^{-1}$ is except for finitely many values of j and k, $\lambda_j^k \pm m (k+\lambda_9)^{-1}= (k+\lambda_0)^2 + \mu_j^k \pm m (k+\lambda_9)^{-1}$. And the $n(j,k)$'th eigenvalue of $U(A^2 +B )U^{-1}$ is just $\nu_{n(j,k)}$, the same as $A^2 + B$. By the minimax principle and the above inequality, we have for almost all j and k
    \[
    \lambda_j^k - m (k+\lambda_9)^{-1} \leq \nu_{n(j,k)} \leq \lambda_j^k + m (k+\lambda_9)^{-1}
    \]
    i.e.
    \[
    |\nu_{n(j,k)} - \lambda_j^k | \leq \frac{m}{k+\lambda_0} \leq \frac{m}{k}
    \]
    Replacing m by a larger number M to allow for the exceptional values of j and k , we just proved the theorem.
    
\end{proof}

\begin{lemma}
     There is a unitary pseudo differential operator U such that $U - I$ and $U(A^2 +B ) - (A^2 + \Bar{B})$ are both of order -1.
\end{lemma}

\begin{proof}
    Let $P = 1/(2\pi i)\int_0^{2\pi}\int_0^t B_s dsdt$. We first prove that $[A,P] = \Bar{B} - B$.
    \begin{align*}
        [A,P]& = \frac{1}{2\pi i }\int_0^{2\pi}\int_0^t[A, B_s] dsdt\\
             & = \frac{1}{2\pi} \int_0^{2\pi}\int_0^t\frac{dB_s}{ds} dsdt\\
             & = \frac{1}{2\pi} \int_0^{2\pi}(B_t - B) dt\\
             & = \Bar{B} - B
    \end{align*}
    where the second step use \eqref{commu}.

    The operator $P$ is a skew adjoint pseudo differential operator of order 0. And we define our $Q$ as
    \[
    Q = \frac{1}{4}(PA^{-1} + A^{-1}P)
    \]
    It is easy to see that $Q$ is a skew adjoint pseudo differential operator of order -1. Next, we need to prove modulo operators of order -1 we have 
    \[
    [A^2,Q] = [A,P] = \Bar{B} - B
    \]
    this relies on the commutator formula 
    \[
    [A^2,A^{-1}P] = 2[A,P] + A^{-1}[[A,P],A]
    \]
    and
    \[
    [A^2,PA^{-1}] = 2[A,P] + [[A,P],A]A^{-1}
    \]
    It is not hard to see that $A^{-1}[[A,P],A]$ and $[[A,P],A]A^{-1}$ are both of order -1.

    Finally we can prove that $U = e^Q$ is uitary and $U -I$ and $U(A^2 +B ) - (A^2 + \Bar{B})$ are both of order -1. Since Q is skew adjoint, U is unitary. By Seeley's functional calculus \cite{Seeley1967ComplexPO}, U is a pseudo differential operator of order 0. Since the 0 order symbol of Q is 0, the principal symbol of U is $e^0 = 1$, so $U -I $ have 0 order symbol equal to zero, then it must have order -1. And we notice that the morphism $T \rightarrow e^QTe^{-Q}$ can be written as a series of operators $\sum_{i=0}^{\infty} \frac{(Ad Q)^n}{n!}$, where $Ad Q$ is a operator such that $(Ad Q)(P) = [Q,P]$. Since $Q $ is of order -1, we know that $Ad Q$ will lowers the order of operators by 2. Then modulo operators of order -1, we have 
    \begin{align*}
        e^Q(A^2+B)e^{-Q} & \equiv A^2 +B + [Q,A^2+B]\\
                         & \equiv A^2 +B -(B - \Bar{B})\\
                         & = A^2 + \Bar{B}
    \end{align*}
    which finish the proof of this lemma.
\end{proof}

After proving Proposition \ref{difference}, we could show example \ref{Zoll Manifold}. First, we decompose the square root of our operator $A^2 + \Bar{B}$ into $A$ plus an order -1 operator Q.
\[
A^2 + \Bar{B} = (A + Q)^2
\]
Q is of order -1 because
\[
\Bar{B}= AQ + QA + Q^2
\]
$\Bar{B}$ is of order 0 and A is of order 1, Q could not have its principal symbol in order 0 because that would make the right hand side have order 1. Also Q will commute with A since we have already prove that A commute with $\Bar{B}$ which could deduce A commute with $A^2 + \Bar{B}$ and then commute with $A + Q$.

We write Q as
\[
Q = \sqrt{A^2 + \Bar{B}} -  A
\]
Multiply both side with $\sqrt{A^2 + \Bar{B}} + A$
\begin{equation}\label{bound}
Q ( \sqrt{A^2 + \Bar{B}}  + A )= \Bar{B}
\end{equation}
Like what we did previously, we could write the the spectrum of $A + Q$ as 
\[
\Lambda_2 = \left\{ \sqrt{\lambda_j^k} = k+\lambda_0 + q_j^k: \quad k \geq 0 \quad \textrm{and} \quad 1 \leq j \leq d_k \right\}
\]
By \eqref{bound}, we know that $q_j^k(k+\lambda_0+\sqrt{(k+\lambda_0)^2+\mu_j^k})$ is bounded since $\Bar{B}$ is a bounded operator. Then we know that
\[
\sqrt{\lambda_j^k} = k+\lambda_0 + O(\frac{1}{k})
\]
which could deduce
\[
\sqrt{\nu_{n(j,k)}} = k+\lambda_0 + O(\frac{1}{k}) 
\]
by proposition \ref{difference}.

From now on, we start to use the classic method to analyze the asymptotic behavior of the joint spectrum of $A^2 + \Bar{B}$, which in brief is just based on the study of Fourier transform of the spectral distribution. The Fourier side satisfy some partial differential equation which could be solved by comstructing some parametrix, and then we could use some tauberian type theorem to get the asymptotic estimates. This method start from Hormander's famous paper ~\cite{Hor68}, in which he did the analysis of the big singularity at zero of the spectral function $\sum e^{-it\mu_k}$ of a first order elliptic operator Q. This will lead to an asymptotic expansion of the following form 
\[
\sum \rho(\mu - \mu_j) \sim (2\pi)^{-n} \sum c_k\mu^{n-1-k}
\]
where $\rho$ is in an appropriate class of Schwartz functions and $c_k$ are the integrals over the cosphere bundle of polynomial expressions in the symbols of our operator Q. 

The above formula proved by Duistermaat and Guillemin in their paper ~\cite{DuistermaatGuillemin} is important for our proof of Lamma \ref{zec}, so we will briefly introduce it here and some details of its proof. Basically we need to inherit from Hormander's analysis of big singularity at zero, use a parametrix constructed by Lax and observe the relation between subprincipal symbol and the second term in the asymptotic expansion.

\begin{theorem} \label{DG1}
    Let Q be a self-adjoint first order pesudo differential operator with its eigenvalues $\{\mu_j\}$. There exists a sequence of real valued smooth densities $\omega_1, \omega_2,\dots$ in manifold X such that for every $\rho \in S(R)$ with $supp(\hat{\rho})$ contained in a sufficiently small neighborhood of 0 and $\hat{\rho} = 1$ in another neighborhood of 0 :
    \[
    \sum_j \rho(\mu - \mu_j) |e_j|^2 \sim (2\pi)^{-n} \sum_{k=0}^{\infty} \omega_k \mu^{n-k-1}
    \]
    for $\mu \rightarrow \infty$ and rapidly decreasing as $\mu \rightarrow -\infty$, asymptotics in the $C^{\infty}(X,\Omega_1)$ topology. Let $S^*_xX$ be the cosphere bundle defined by $S^*_xX = \{ \xi \in (T_x X)^* : q(x,\xi) = 1\}$. Here we have 
    \[
    \omega_0(x) = vol(S^*_xX) , \quad \omega_n(x) = 0 
    \]
    and 
    \[
    \omega_1(x) = (1-n) \int_{S^*_xX} sub(Q).
    \]
\end{theorem}
Here $sub(Q)$ is the subprincipal symbol of Q which is defined by the following formula locally.
\begin{equation*}
    sub (Q )= q_{n-1} - (2i)^{-1} \sum_{j-1}^n \frac{\partial^2 q_n}{\partial x_j\partial \xi_j}
\end{equation*}
where Q is a pseudo differential operator of order n and $q_{n-1}$ is the $n-1$ order symbol of Q. Its subprincipal symbol consists of two parts, the second term in its total symbol (which is $q_{n-1}$) and the sum of derivatives of the principal symbol. We know that the principal symbol can be invariantly defined on the whole manifold, but the total symbol can't. However, by adding these derivatives, we can make the subprincipal symbol invariantly defined globally. (See ~\cite{FIO2} ~\cite{FIO1} for more discussion about this.)

About subprincipal symbol, we have this formula:
\begin{equation} \label{sub}
    sub(P^{\alpha}) = \alpha p^{\alpha -1} sub(P)
\end{equation}
for all complex number $\alpha$, where $p$ here is the principal symbol of P. (See \cite{DuistermaatGuillemin} for the proof)

Consider the Laplace-Beltrami operator $\Delta$ on manifold, we claim that the subprincipal symbol of $\Delta$ is 0. This seems like a widely known result, but I did not find easy proof by searching on the Internet, so I decide to prove it here.
\begin{lemma}\label{sublap}
    On any manifold M, let $\Delta$ be the standard Laplace-Beltrami operator on it, then we have
    \begin{equation}
        sub(\Delta) = 0
    \end{equation}
\end{lemma}

\begin{proof}
    First we know that in local coordinates, we can represent the symbol of Laplace Beltrami operator as the following:
    \[
    \sigma(\Delta) = \sum_{j,k=1}^n g^{jk}(x) \xi_j\xi_k
    \]
    it only contains the principal symbol, which is the second order symbol. The first order symbol is 0. So the subprincipal symbol only contains the derivatives part, which is.
    \[
    sub(\Delta) = -(2i)^{-1} \sum_{i=1}^n \frac{\partial^2}{\partial x_i\partial \xi_i} (\sum_{j,k=1}^n g^{jk}(x) \xi_j\xi_k)
    \]
    We can directly prove that this formula is equal to zero, or alternatively, we could apply standard differential geometry techniques to select a normal coordinate system. In such a system, all first-order derivatives of the metric tensor vanish at a specific point. Since we can choose normal coordinates at any point on our manifold, this allows us to demonstrate that the subprincipal symbol is zero at that point. Given that the subprincipal symbol is invariantly defined, it possesses properties similar to those of tensors, which implies that it must be zero globally across the entire manifold. This approach is a widely used technique in modern geometry.
\end{proof}

Since $\Delta$ and $\Delta +c$ have the same principal symbol, and the first order symbol of them are both zero. So they will just have the same subprincipal symbol, which just gives us $sub(\Delta +c) = 0$. By the \eqref{sub}, just choose $\alpha = 1/2$, we can obtain the following corollary, which will help us compute the spectral asymptotics of A:

\begin{corr}
    Recall that we define our operator $A_0$ just as this: $A_0 = \sqrt{\Delta + c}$. And we define our operator as this: $A = A_0 - T$. Here T is a pseudo differential operator of order -1 which is already defined. The subprincipal symbol of A and $A_0$ are both zero.
    \[
    sub(A_0) = 0 \quad and \quad sub(A) = 0
    \]
\end{corr}

\begin{proof}
    Just by definition, the subprincipal symbol is computed by using the principal symbol and the second term in the total symbol. So we have 
    \[
    sub(\Delta + c) = sub(\Delta) = 0
    \]
    by lemma \ref{sublap}. Then we obtain
    \[
    sub(A_0) = sub (\sqrt{\Delta + c}) = 0
    \]
    by \eqref{sub}. Finally we get: 
    \begin{equation}
        sub(A) = sub(A_0) = 0
    \end{equation}
    since A and $A_0$ are both pseudo differential operator of order 1, and T is a pseudo differential operator of order -1. So T won't contribute to the computation of $sub(A)$.
\end{proof}

\begin{proof}

Now we start to present the sketch proof of Theorem \ref{DG1}. First, let $U(t,x,y)$ be the kernel of $e^{-itQ}$. This operator could be realized as a solution operator of the following equaiotn
\[
i^{-1} \frac{\partial}{\partial t} u +Q u = 0,  \qquad u(0,x) = u_0(x)
\]
So we could use parametrix to approximate the kerneal $U(t,x,y)$.
Actually this kernel can be written as a locally finite sum of integrals of the following form:
\begin{equation} \label{integral}
    I(t,x,y) = (2\pi)^{-n} \int e^{i\Phi(t,x,y,\zeta)} a(t,x,y,\zeta) d\zeta
\end{equation}
See ~\cite{soggefica} or other classical materials for more details.
Here $\Phi(t,x,y,\zeta)$ is a non-degenerate phase function introduced by Hormander in ~\cite{FIO1}. And :
\begin{equation}
    d_{\zeta} \Phi = 0 \Rightarrow ((t,d_t\Phi),(x,d_x\Phi),(y,-d_y\Phi)) \in C
\end{equation}
C is the canonical relation associated with the Fourier Intergral Operator $U=e^{-itQ}$. And $a(t,x,y,\zeta)$ is a symbol of order 0.

For small t, $U(t,x,y)$ can locally be represented by integral of the form \eqref{integral} with a special type of phase function invented by Hormander in~\cite{Hor68}:
\begin{equation}
    \Phi(t,x,y,\eta) = \psi(x,y,\eta) - t \cdot q(y,\eta)
\end{equation}
Where q is the symbol of the operator $Q$ and $\psi$ is as a function of the first variable, the local solution near diagonal of the Cauchy problem:
\begin{align*}
    q(x,d_x\psi(x,y,\eta)) = & q (y,\eta),\\
    \psi(x,y,\eta) = 0 \quad when \quad & \langle x-y,\eta \rangle = 0 \\
    d_x\psi(x,y,\eta) = \eta \quad for \quad & x=y
\end{align*}
This equation is the so called eikonal equation in geometric optics. The phase function here is different from the simple phase function appears in Hadamard parametrix construction. Usually the parametrix constructed using this $\Phi$ may be called as Lax parametrix. What makes $\Phi$ special is that it is linear in t with $q(y,\eta)$ not related to x and $\psi(x,y,\eta)$ not related to t. This property somehow makes the composition formula or the stationary phase easier to compute. 

Remember that $U(t,x,y)$ can be written as
\[
U(t,x,y) = \sum_{j=1}^{\infty} e^{-it\mu_j} \cdot  e_j(x) \cdot  \overline{e_j(y)}
\]
And the trace of U, which is just integrate over the diagonal of the kernel, could be written as
\[
Tr( U(t)) = \int_{x\in X} \sum_{j = 1}^{\infty} e^{-it\mu_j}|e_j(x)|^2 = \sum_{j=1}^{\infty} e^{-it\mu_j}
\]
The trace of U is just the fourier transform of the following spectral distribution
\[
\sigma(\mu) = \sum_{j=1}^{\infty} \delta(\mu - \mu_j)
\]

We can use $\hat{\rho}$ to cut off our kernel which could help us rule out the other singularities different from zero. Then we can do inverse fourier transform to get back to $\rho(\mu-\mu_j)|e_j|^2$. Substituting $\eta = \mu \cdot \omega \cdot \tilde{\eta}$, where $\tilde{\eta}$ satisfy $q(y,\tilde{\eta}) =1$, and $\omega > 0$ . We use the method of stationary phase to get,
\[
(2\pi)^{-1} \int e^{i\mu t} \hat{\rho}(t) U(t,y,y) dt = (2\pi)^{-n} \mu^{n-1}\int_{S^*_x X} b(y,\tilde{\eta},\mu) d\tilde{\eta}
\]
with
\[
b(y,\tilde{\eta},\mu) \sim \frac{1}{r!}(i^{-1}\frac{\partial^2}{\partial t\partial \omega})^r[\hat{{\rho}}(t) a (t,y,y,\mu\omega\tilde{\eta}) \cdot \omega^{n-1}]_{t=0 \atop \omega =1}
\]
as $\mu \rightarrow \infty$, uniformly in $(y,\tilde{\eta}) \in S^*X$.

Here the integral over the unbounded t and $\omega$ domain is fine because we can do this integral first in a neighborhood of $\omega =1$ and $t = 0$, which is a bounded domain. Then we can prove the other part is rapidly decreasing since we have no stationary points for our phase function in that domain.

This gives us the formula for $\omega_k$:
\begin{align*}
    \omega_k(y) &= \int_{S^*_x X} a_{-k}(0,y,y,\tilde{\eta}) d\eta + \sum_{r=1}^k(n-k+r-1) \cdots (n-k)\cdot \\
    &\int_{S^*_x X} \frac{1}{r!}(i^{-1}\frac{\partial}{\partial t})^r a_{-k+r}(0,y,y,\tilde{\eta}) d\tilde{\eta}
\end{align*}

In order to compute the $\omega_k$, observe that we have the initial condition:
\begin{equation}\label{ic}
    (2\pi)^{-n} \int e^{i\psi(x,y,\eta)} a (0,x,y,\eta) d\eta = U(0,x,y) = \delta(x - y)
\end{equation}
Recall that $\psi(x,y,\eta)$ is defined as the solution to an ODE (eikonal equation). It is easy to check that $\psi(x,y,\eta)$ and $\langle x-y,\eta \rangle$ will define the same Lagrange manifold. Also $d^2_{\eta}\psi = d^2_{\eta} \langle x-y,\eta \rangle= 0$ at $x = y$. We could choose a locally homogeneous change of coordinates $\eta = \eta(x,y,\zeta)$ such that $\eta(y,y,\zeta) = \zeta$ and :
\begin{equation}
    \psi(x,y,\eta(x,y,\zeta)) = \langle x-y,\zeta \rangle
\end{equation}
according to Hormander ~\cite{FIO2} and ~\cite{FIO1}. Subtituting this in the left hand side of the above formula \eqref{ic} we see that $U(0) = I$ if we choose 
\begin{equation}
    a(0,x,y,\eta(x,y,\zeta)) = 1/|det (d_{\zeta} \eta(x,y,\zeta))|
\end{equation}

In particular, we obtain $a_0(0,y,y,\eta) =1 $ and $a_{j}= 0$  for all the other $j<0$ by letting $x = y$ and $\eta(y,y,\zeta)= \zeta$ in the above equation.

Applying $i^{-1} \partial/\partial t + Q$ to \eqref{integral} under the integral sign amounts to replacing the symbol a by a new symbol $a'$ such that
\begin{align*}
a'(t,x,y,\eta) &\sim \frac{1}{i}\frac{\partial a}{\partial t}(t,e,y,\eta) -q(y,\eta)\cdot a(t,x,y,\eta)\\
&+ \sum_{k} \frac{1}{k!} (i^{-1}\sum_{j}\frac{\partial^2}{\partial\tilde{x_j}\partial\tilde{\xi_j}}[Q(x,\chi(\tilde{x},x,y,\eta)+\tilde{\xi})a(t,x+\tilde{x},y,\eta)]_{\tilde{x}=0 \atop \tilde{\xi} =0}
\end{align*}
for $\eta \rightarrow \infty$, here we have written
\begin{equation}
    \psi(x+\tilde{x},y,\eta) = \psi (x,y,\eta) + \langle \tilde{x}, \chi(\tilde{x},x,y,\eta) \rangle
\end{equation}

The equation
\[
a'_{-j}(t,x,y,\eta) = 0
\]
is a first order linear partial differential equations for $a_{-j}$ involving only the $a_{-k}$ for $k < j$. With the initial condition, these "transport equations" could be solved uniquely in small time. This is just the parametrix construction process. The computation here is from Duistermatt and Guillemin \cite{DuistermaatGuillemin}. You could also see Hormander \cite{Hor68}. With all of this, we just get a local description of the solution U of $(i^{-1}\partial /\partial t +Q)U = 0$, $U(0) = I$, modulo integral operators with smooth kernels.

Actually all the $\omega_k$ could be computed using the above transport equations. Although, for the higher order terms, it is too complicated to compute, thanks to our subprincipal symbol, we can get $i^{-1} \partial a_0/\partial t (0,y,y,\eta) = -sub(Q)(y,\eta)$. With this and the above formula for $\omega_k$, we can get $\omega_1 =(1-n)\int_{S^*_x X}sub(Q)$.
\end{proof}

Our main goal in this chapter is to analyze the spectrum of $\Delta = A^2 + B$, but by the following observations, we know that there are not too much difference between $A^2 +B$ and $A^2 + \Bar{B}$.

The spectrum of $A^2 + \Bar{B}$ is $\Lambda = \{(k+\lambda_0)^2 + \mu_j^k\}$, which will cluster around $(k+\lambda_0)^2$. And by \eqref{difference} we can easily see that the eigenvalues of $A^2 + B$ will also cluster around $(k+\lambda_0)^2$, we call this k'th cluster. We define the sum of characteristic measures in k'th cluster of these two operators by:
\[
\phi_k (\lambda) = \sum_{j=1}^{d_k} \delta( \lambda - \nu_j^k)
\]
and
\[
\psi_k (\lambda) = \sum_{j=1}^{d_k} \delta( \lambda - \mu_j^k)
\]

It is easy to see that the following is a corollary of proposition \ref{difference} :
\begin{align*}
    \frac{1}{d_k} \left\langle \phi_k - \psi_k , \rho \right\rangle &= \frac{1}{d_k} \sum_{j=1}^{d_k}[\rho(\nu_j^k) - \rho(\mu_j^k)] \\
    &\leq \frac{1}{d_k} ||\rho|| \sum_{j=1}^{d_k}(\nu_j^k- \mu_j^k)\\
    &\leq \frac{1}{d_k} ||\rho|| d_k \frac{M}{k} = ||\rho|| \frac{M}{k}
\end{align*}
Here we could choose a uniform constant for $\rho$ in a bounded set of $C_0^{\infty}(\R)$ with $C^1$ topology. The norm of $\rho$ is $C^1$ norm. The first step is by mean value theorem and the second step is by proposition \ref{difference}.

Although our main goal is to study the asymptotic behavior of the $d_k$, which is the number of eigenvalues in the k'th cluster, we still need some previous knowledge. Let us first consider the inner product of characteristic measure in k'th cluster with a smooth function $\rho$.
\begin{equation}\label{fc}
\left \langle \psi_k , \rho\right \rangle = \sum_{j=1}^{d_k} \rho (\mu_j^k)
\end{equation}
These numbers can be realized as Fourier coefficients of a distribution on the circle. It is well known that we have a coarse estimate about eigenvalues $d_k = O(k^{d-1})$, (See ~\cite{soggefica} for example), so $\left \langle \psi_k , \rho\right \rangle$ will just have polynomial growth. Then we can define a $2\pi$-periodic distribution $\theta_{\rho}$ as a sum of the following Fourier series:
\begin{align}
    \theta_{\rho} (t) &=\sum_{k=0}^{\infty} \left \langle \psi_k , \rho\right \rangle  e^{-ikt}\\
    &=\sum_{k=0}^{\infty} \sum_{j=1}^{d_k} \rho (\mu_j^k) e^{-ikt}
\end{align}
Actually we have done a centralization process here to omit $(k+\lambda_0)^2$ for each term $\psi_k$. This will make the support of these distributions $\psi_k$ in a compact set centered at 0, which could help us to use the globally test function $\rho \in C^{\infty}(\R)$ and count the number of eigenvalues easier than the usual settings. This is the main difference comparing with the standard analysis of the distribution $\sigma(\mu) = \sum_{j=1}^{\infty} \delta(\mu -\mu_j)$ where $\mu_j$ are all the eigenvalues. In that case, the distribution has unbounded support which allows us only to use compact supported test functions. The reason allow us to do this is that the bicharacteristic flow of laplacian on Zoll manifold is periodic with the same period, which can help us find a periodic propagator $e^{-itA}$. Because of this, we could somehow compactify the timeline into a circle. That is why we use fourier coefficients here instead of fourier transform in Theorem \ref{DG1}.

Let us define our operators $S(t)$ according to $\theta_{\rho} (t)$:
\begin{equation}
    S(t) = \rho(\Bar{B}) e^{-it(A-\lambda_0I)}
\end{equation}
Here we just use the standard spectral theory definition, then $S(t)(e_j^k) = \rho(\mu_j^k) e^{-ikt} e_j^k$ since $e_j^k$ are the common basis of both A and $\Bar{B}$. Now we can formally write $\theta_{\rho}(t)$ as traces of our newly defined operator:
\begin{equation}
    \theta_{\rho}(t) = Tr (S(t)) = Tr (\rho(\Bar{B}) e^{-it(A-\lambda_0I)})
\end{equation}
The trace here is just a integral in the diagonal of the kernel. To justify this rigorously, you could just apply both side to a test function and integrate by parts for enough times. See ~\cite{DuistermaatGuillemin} for details.

Our operator consists of two parts, the amplitude part $\rho(\Bar{B})$ and the oscillatory part $e^{-iAt}$. The oscillatory part is an elliptic operator whose associated canonical transformation is the map $f_t: T^{*}X/0 \xrightarrow{f_t} T^{*}X/0$. Since Z is our Zoll manifold , and we have assumed the bicharacteristic flow $f_t$ is simply periodic with period $2\pi$, we can make the following assumptions from now:
\begin{equation}
    f_t(x) \ne f_{\tau} (x) , \quad for \quad t \ne \tau
\end{equation}
where $t,\tau \in (0,2\pi)$ and x is any point on manifold Z.

Now we consider the amplitude part $\rho (\Bar{B})$. Recall that the operator $\Bar{B}$ is defined by $\Bar{B} = \frac{1}{2\pi} \int_0^{2\pi} B_t dt$. We can prove that the principal symbol of $\Bar{B}$ given by the following formula:
\begin{equation}\label{principal}
    \sigma_{\Bar{B}} = \frac{1}{2\pi} \int_0^{2\pi} \sigma_B \circ f_t dt
\end{equation}
This is easy to prove. Since $B_t = e^{itA}B e^{-itA}$, by Egorov's theorem, $B_t$ is a pseudo differential operator of order 0 whose principal symbol is just $\sigma_{B_t} = \sigma_B \circ f_t $. Then we can prove formula \eqref{principal} just by integration over t. With this simple results, we can get the principal symbol of $\rho(\Bar{B})$ just by the following proposition:

\begin{proposition}
    For all function $\rho \in C^{\infty}_0(R)$, $\rho(\Bar{B})$ is a pseudo differential operator of order 0 whose principal symbol is given by
    \[
    \sigma_{\rho(\Bar{B})} = \rho \circ \sigma_{\Bar{B}}
    \]
\end{proposition}
We don't want to prove it here. For those who have interests in this, see ~\cite{Wein77} or the other classical material about microlocal analysis.

We now go back to the analysis of $Tr(S(t))$. $S(t)$ have two parts, first is $\rho(\Bar{B})$, which is a pseudo differential operator of order 0. And the other part is $e^{it\lambda_0}e^{-itA}$, which is a one-parameter group of operators. We can consider it as a single Fourier integral operator (FIO) of order -1/4 from $C^{\infty}(X)$ to $C^{\infty}(X \times R)$. Then according to calculus of FIO (see ~\cite{FIO2} for details), $S(t)$, as a composition of $e^{it\lambda_0}e^{-itA}$ and $\rho(\Bar{B})$, is also a FIO of order -1/4 from $C^{\infty}(X)$ to $C^{\infty}(X \times R)$ associated with the same canonical relation as $e^{-itA}$, the standard wave propagator.

Let $S(t,x,y)$ be the Schwartz kernel of our operator $S(t)$, since we have already assumed the bicharacteristic flow $f_t$ does not have fixed points. So the trace $\theta_{\rho} (t) = Tr (S(t)) = \int_X S(t,x,x) dx$ is a smooth function on $(0,2\pi)$. Then we are left with the only big singularities at 0. By analyzing this singularity we can determine the asymptotic behavior of the Fourier coefficients \eqref{fc}, which is the key part of the proof. Actually the analysis of this big singularity start from Hormander's classic ~\cite{Hor68} in which he did the analysis of big singularity at 0. And in Duistermaat and Guillemin's paper ~\cite{DuistermaatGuillemin}, they proved theorem \ref{DG1} with the idea of Hormander and some further computations like the subprincipal symbol.  

We use the same method in proving Theorem \ref{DG1} to deal with the big singularity at $t=0$ of $S(t,x,x)$. We also write $S(t,x,y)$ as oscillatory integral, with the same phase fucntion $\Phi(t,x,y,\eta)$ introduced by Hormander. However, we need to incorporate the factors $e^{it\lambda_0}$ and $\rho(\Bar{B})$ to our integral since $S(t)$ is a composition of $\rho(\Bar{B})$, $e^{it\lambda_0}$ and $e^{-itA}$. On the principal symbol level, the amplitude is simply multiplied by $e^{it\lambda_0}\sigma_{\rho(\Bar{B})}$. 

With all of the preparation above, we can state the propositon proved by Weinstein in ~\cite{Wein77}.

\begin{proposition} \label{we77}
    The Fourier coefficients of $\theta_{\rho}(t)$ have an asymptotic expansion :
    \begin{equation}
        \langle \psi_k, \rho \rangle \sim (2\pi)^{-n} \sum_{\nu = 0}^{\infty} \langle c_{\nu}, \rho \rangle k^{n-\nu -1}
    \end{equation}
    where each coefficient $\langle c_{\nu}, \rho \rangle$ is linear in $\rho$ and is the integral over the unit cosphere bundle $\sigma_A^{-1}(1)$( in Theorem \ref{DG1}, we use $S^*_x X$ to represent this bundle, and operator A is changed to be Q ) of a polynomial in the complete symbols of A and $\rho(\Bar{B})$. In particular, $\langle c_0, \rho \rangle = \int_{\sigma_A^{-1}(1)} \sigma_{\rho(\Bar{B})}$.
\end{proposition}

In terms of the original operator B instead of $\Bar{B}$, the distribution $c_0$ is defined by 
\[
\langle c_0,\rho \rangle = \int_{\sigma_A^{-1}(1)} \rho \circ [\frac{1}{2\pi}\int_0^{2\pi}\sigma_B \circ f_t dt]
\]
In other words, $c_0$ is the push-forward to R, under the averaged symbol $\sigma_{\Bar{B}} = \frac{1}{2\pi}\int_0^{2\pi}\sigma_B \circ f_t dt$, of the natural measure on cosphere bundle $\sigma_A^{-1}(1)$. The support of $c_0$ is contained in the range of $\sigma_{\Bar{B}}$, so $\langle c_0,\rho \rangle $ is defined for all $\rho \in C^{\infty}(R)$. In particular, $\langle c_0,1 \rangle $ is the volume of $\sigma_A^{-1}(1)$.

The higher order terms $\langle c_{\nu}, \rho \rangle$ is a little bit complicated to compute because of the factor $\rho(\Bar{B})$. However, to prove Lemma \ref{zec}, we only need to count the number of eigenvalues in the k'th cluster, for which we don't have to get the full information of Fourier coeffcients $\langle \psi_k, \rho \rangle$ for all $\rho$.  We only need to consider $\rho =1$ which will give us :
\begin{equation}
    d_k = \langle \psi_k ,1  \rangle = (2\pi)^{-n} \sum_{\nu = 0}^{1} \langle c_{\nu}, 1\rangle k^{n-\nu -1} + O(k^{n-3}).
\end{equation}
We have already computed the first term:
\[
\langle c_0,1  \rangle = \int_{\sigma_A^{-1}(1)} 1 = Vol(\sigma_A^{-1}(1))
\]
The second term is a little bit tricky here, since we don't want to repeat the hard and tedious computation which have been done by Hormander, Duistermaat and Guillemin. We want to somehow use Theorem \ref{DG1} instead. So let us recall the settings of Theorem \ref{DG1} and Proposition \ref{we77}. Theorem \ref{DG1} starts from analyzing the traces of $U(t,x,y)$ where $U = e^{-itQ}$:
\begin{equation}
    Tr (U(t)) = \int_{x\in X} \sum_{j=1}^{\infty} e^{-it\mu_j} |e_j|^2 = \sum_{j=1}^{\infty} e^{-it\mu_j}
\end{equation}
while Proposition \ref{we77} first considers the trace of this operator $\rho(\Bar{B})e^{-it(A-\lambda_0I)}$:
\begin{equation}
    Tr(\rho(\Bar{B})e^{-it(A-\lambda_0I)}) = \theta_{\rho}(t) = \sum_{k=0}^{\infty} \langle \psi_k,\rho \rangle e^{-ikt}
\end{equation}

Since now we are computing the number of eigenvalues, we can replace our $ A^2 + \Bar{B}$ by just $A^2$, i.e. let $\Bar{B} = 0$, which won't change the number $d_k$. Then we may prove the following formula of $\langle c_1,1 \rangle$ as what we did in Theorem \ref{DG1}:
\begin{equation}
    \langle c_1,1 \rangle = (1-n) \int _{\sigma_A^{-1}(1)} sub(A-\lambda_0I)
\end{equation}

We have already proved that the subprincipal symbol of $A = A_0 -T = \sqrt{\Delta + cI} - T$ is 0. However, by a translation of $\lambda_0I$, the subprincipal symbol of $A - \lambda_0I$ is not zero. That is why we need to make a change of variables $t = k + \lambda_0$ as in the sphere spectrum case \eqref{polynomial}. Finally we get for zoll manifold , $\lambda_0 = \alpha /4$, where $\alpha$ is the so called Maslov index. For sphere case , this Maslov index is equal to $2(n-1)$ which is the same as \eqref{polynomial}, $t= k + (n-1)/2$. 

In general, we have the following for a Zoll manifold of dimension n:
\begin{equation}
    d_k = (2\pi)^{-n} vol(\sigma_A^{-1}(1)) t^{n -1} + O(t^{n-3})
\end{equation}
where $A$ is a operator defined by $A = A_0 -T = \sqrt{\Delta + cI} - T$ and t is equal to $k + \alpha/4$. Now we have finished the proof of Lemma \ref{zec}.

\newpage

\newsection{Proof of Theorem \ref{zollweyl}}

Let $M_1$ and $M_2$ both be compact Riemannian manifolds without boundary with $d_i = \dim M_i$ for $i = 1,2$. Their product $M = M_1 \times M_2$ is again a boundaryless, compact Riemannian manifold endowed with the product metric. The Laplace-Beltrami operator on $M$ is
\[
    \Delta_M = \Delta_{M_1} \otimes I + I \otimes \Delta_{M_2}.
\]
If $e_1$ and $e_2$ are eigenfunctions on $M_1$ and $M_2$, respectively, with
\[
    \Delta_{M_i} e_i = -\lambda_i^2 e_i \qquad \text{ for $i = 1,2$},
\]
then their tensor $e_1 \otimes e_2$ is an eigenfunction of the Laplacian $\Delta_M$ with
\[
    \Delta_M e_1 \otimes e_2 = (\Delta_{M_1} e_1) \otimes e_2 + e_1 \otimes (\Delta_{M_2} e_2) = -(\lambda_1^2 + \lambda_2^2)e_1 \otimes e_2.
\]
If $e_1$ and $e_2$ are drawn from a Hilbert basis of eigenfunctions on $M_1$ and $M_2$, respectively, then the tensors $e_1 \otimes e_2$ form a Hilbert basis for $L^2(M)$. We can construct the spectrum $\Lambda$ on $M$ from the spectra $\Lambda_1$ and $\Lambda_2$ for $M_1$ and $M_2$ by
\[
    \Lambda = \left\{\sqrt{\lambda_1^2 + \lambda_2^2} : \lambda_1 \in \Lambda_1 \text{ and } \lambda_2 \in \Lambda_2 \right\}
\]
with multiplicities
\[
    \mu(\lambda) = \sum_{\substack{(\lambda_1,\lambda_2) \in \Lambda_1 \times \Lambda_2 \\ \lambda_1^2 + \lambda_2^2 = \lambda^2}} \mu_1(\lambda_1) \mu_2(\lambda_2),
\]
where here $\mu_1$ and $\mu_2$ are the respective multiplicities for $\Lambda_1$ and $\Lambda_2$. The Weyl counting function for $M$ can be written
\[
    N(\lambda) = \sum_{\substack{(\lambda_1,\lambda_2) \in \Lambda_1 \times \Lambda_2 \\ \lambda_1^2 + \lambda_2^2 \leq \lambda^2}} \mu_1(\lambda_1) \mu_2(\lambda_2).
\]
A similar formula holds for the Weyl counting function if $M$ is an $n$-fold product $M_1 \times \cdots \times M_n$ of compact, boundaryless Riemannian manifolds with respective spectra and multiplicities $\Lambda_i$ and $\mu_i$ for each $i$. Namely,
\begin{equation}\label{product counting formula}
    N(\lambda) = \sum_{\substack{(\lambda_1,\ldots,\lambda_n) \in \Lambda_1 \times \cdots \times \Lambda_n \\ \lambda_1^2 + \cdots + \lambda_n^2 \leq \lambda^2}} \prod_{i = 1}^n \mu_i(\lambda_i).
\end{equation}

Now to specify to the case of Zoll manifolds. Let $M = Z^{d_1} \times Z^{d_2} \times \cdots \times Z^{d_n}$. If $M$ contains any $Z^1$ factors, we gather them on the rightmost side of the product. That is,
\[
    M = Z^{d_1} \times \cdots \times Z^{d_k} \times \underbrace{Z^1 \times \cdots \times Z^1}_{\text{$n-k$ times}} \simeq Z^{d_1} \times \cdots \times Z^{d_k} \times \T^{n-k}
\]
where $\T^{n-k} = \R^{n-k}/2\pi \Z^{n-k}$ is the $(n-k)$-dimensional torus. We also write the dimension multiindex
\[
    d = (d_1,\ldots,d_k,\underbrace{1,\ldots,1}_{\text{$n-k$ times}}) \qquad \text{ with } d_1,\ldots, d_k \geq 2.
\]

Here we notice that $Z^1$ is just the same as $S^1$, this is not hard to see because there does not exist one dimensional closed manifold other than $S^1$ and there are no intrinsic different metrics on $S^1$ other than standard metric. 

If all the Zoll manifold are just round spheres, we can write the Weyl counting function for M as $N^s(\lambda)$:
\begin{equation} \label{sphere product counting function}
    N^s(\lambda) = \sum_{\substack{m \in \Z_{\geq 0}^k \times \Z^{n-k} \\ |m + y|^2 \leq \lambda^2 + |y|^2}} \prod_{i = 1}^k \left( \binom{m_i + d_i}{d_i} - \binom{m_i + d_i - 2}{d_i} \right)
\end{equation}
by \eqref{product counting formula}, \eqref{sphere spectrum} and \eqref{torus counting function}. Where
\[
    y = \left( \frac{\alpha_1}{4} , \ldots , \frac{\alpha_k}{4} , 0, \ldots, 0 \right).
\]

\begin{exmp}\label{flat torus}[The Flat Torus]
The flat torus admits a Hilbert basis of eigenfunctions in the form of exponentials
\[
    e_m(x) = (2\pi)^{-d/2} e^{-i\langle x , m \rangle} \qquad \text{ for $m \in \Z^d$}
\]
each satisfying
\[
    \Delta e_m = -|m|^2 e_m.
\]
Hence we have spectrum and multiplicities,
\[
    \Lambda = \{ -|m|^2 : m \in \Z^d \} \qquad \text{ and } \qquad \mu(\lambda) = \#\{ m : |m| = \lambda\}.
\]
The Weyl counting function then just counts the number of integer lattice points lying in the ball of radius $\lambda$,
\begin{equation}\label{torus counting function}
    N(\lambda) = \#\{m \in \Z^d : |m| \leq \lambda\}.
\end{equation}
\end{exmp}

It was proved in 1950 by Hlawka \cite{Hlawka1950}, using the Poisson summation formula that 
\[
    N(\lambda) = |B_d| \lambda^{d} + O(\lambda^{d-1 - \frac{d-1}{d+1}}).
\]
The remainder bound is not sharp and the exponent has been improved little by little over the decades. In dimensions $d \geq 5$, there is a sharp remainder of $O(\lambda^{d-2})$ (see e.g. \cite{Fricker82}). In dimension $d = 2$, the conjectured $O(\lambda^{\frac12 + \epsilon})$ remainder for all $\epsilon > 0$ remains open, with the best current exponent due to Bourgain and Watt \cite{BourgainWatt}. In three dimensions, the conjectured $O(\lambda^{1+\epsilon})$ bound is also open, with the best known exponent due to Heath-Brown \cite{HeathBrown}. See also \cite{huxley1996area, EkkehardKrätzel1992} for a thorough description of the distribution of lattice points in the ball, and, consequently, the Weyl law results on the torus.

If we consider Zoll manifolds instead of round spheres, our eigenvalues just lie in the small rectangles around the lattice points. We don't know the exact distribution of those eigenvalues but we will prove these small perturbation would have influence to our Weyl counting function only in the $O^{|d| -2}$ level.

First let m be our lattice point
\[
 m = \left( m_1,m_2,\cdots, m_k, m_{k+1}, \cdots, m_n \right)
\]
and $\delta_m$ to be the size of cluster
\[
\delta_m = \left(\frac{c_1}{m_1}, \cdots, \frac{c_k}{m_k},0,\cdots, 0\right)
\]
Here $c_i$ is just some constants depend on Zoll manifold $Z_i$. We define $Q_m$ to be the small cube where the eigenvalues of our product Zoll manifolds lie in. Then $Q_m$ is the following
\begin{equation}
Q_m = \left\{ x :m + y - \delta_m \le x \le  m+ y + \delta_m \right\}
\end{equation}
by \eqref{Zollcluster}. It is not hard to see that the length to the origin for points inside $Q_m$ will vary from $|m+y-\delta_m|$ to $|m+y+\delta_m|$. If we set that $|m+y| \sim \lambda$, then we have these length in between $\lambda - c/\lambda$ and $\lambda + c/\lambda$ for some constant c. We could check this by some easy computation:
\begin{align*}
    |m+y| - |m + y - \delta_m| = \sqrt{(m_1+\frac{\alpha_1}{4})^2+\cdots+(m_k+\frac{\alpha_k}{4})^2+m_{k+1}^2+\cdots+m_n^2} -\\
    \sqrt{(m_1+\frac{\alpha_1}{4}-\frac{c_1}{m_1})^2+\cdots+(m_k+\frac{\alpha_k}{4}-\frac{c_k}{m_k})^2+m_{k+1}^2+\cdots+m_n^2}\\
    \approx \frac{2\sum_{j=1}^k( c_i+ \frac{\alpha_ic_i}{4m_i}) - \sum_{j=1}^k \frac{c_i^2}{m_i^2}}{2 \sqrt{(m_1+\frac{\alpha_1}{4})^2+\cdots+(m_k+\frac{\alpha_k}{4})^2+m_{k+1}^2+\cdots+m_n^2}}
\end{align*}
this expression will give you a fraction with a $O(\lambda)$-size denominator and a $O(1)$-size numerator as $\lambda$ grows.

Now we could write the Weyl counting function for the product of Zoll manifolds as 
\begin{equation}
N(\lambda) = \sum_{\substack{m \in \Z_{\geq 0}^k \times \Z^{n-k} \\ |m + y+\delta_m|^2 \leq \lambda^2 }} \prod_{i = 1}^k P_{Z_i} (m_i +\frac{\alpha_i}{4}) + \sum_{\substack{m \in \Z_{\geq 0}^k \times \Z^{n-k} \\ |m + y -\delta_m|^2 < \lambda^2 <|m + y +\delta_m|^2 }}  \sharp \{\mu \in Q_m \cap B_\lambda \}
\end{equation}
where the first part is just counting eigenvalues in the small rectangles which are totally inside the ball $B_{\lambda}$, and the second part is to count those eigenvalues in the rectangles which intersect the boundaries of $B_{\lambda}$.

We denote the first and second part by $\rom{1},\rom{2}$. By the above observation of the radius of points inside $Q_m$, we have the following estimates
\begin{align*}
    \rom{2} & \le \sum_{\substack{m \in \Z_{\geq 0}^k \times \Z^{n-k} \\ |m + y -\delta_m|^2 < \lambda^2 <|m + y +\delta_m|^2 }}  \sharp \{\mu \in Q_m  \} \\
    & = \sum_{\substack{m \in \Z_{\geq 0}^k \times \Z^{n-k} \\ |m + y -\delta_m|^2 < \lambda^2 <|m + y +\delta_m|^2 }} \prod_{i = 1}^k P_{Z_i} (m_i +\frac{\alpha_i}{4}) \\
    & \le \sum_{\substack{m \in \Z_{\geq 0}^k \times \Z^{n-k} \\ \lambda - c/\lambda < |m+y| < \lambda + c/\lambda }} \prod_{i = 1}^k P_{Z_i} (m_i +\frac{\alpha_i}{4})
\end{align*}
and 
\begin{equation*}
    \sum_{\substack{m \in \Z_{\geq 0}^k \times \Z^{n-k} \\ |m+y| \le \lambda - c/\lambda }} \prod_{i = 1}^k P_{Z_i} (m_i +\frac{\alpha_i}{4}) \le  \rom{1} \le \sum_{\substack{m \in \Z_{\geq 0}^k \times \Z^{n-k} \\ |m+y| \le \lambda }} \prod_{i = 1}^k P_{Z_i} (m_i +\frac{\alpha_i}{4})
\end{equation*}

Now we define $\hat{N}(\lambda)$ to be 
\begin{equation}\label{ZLC}
    \hat{N}(\lambda) : = \sum_{\substack{m \in \Z_{\geq 0}^k \times \Z^{n-k} \\ |m+y| \le \lambda }} \prod_{i = 1}^k P_{Z_i} (m_i +\frac{\alpha_i}{4})
\end{equation}
If we can prove the following Lemma, along with the above estimates for $\rom{1},\rom{2}$, we can finish the proof Theorem \ref{zollweyl}.

\begin{lemma}\label{ZollLatticeCount}
    Let $\hat{N}(\lambda)$ defined by \eqref{ZLC}. Then we have 
    \begin{equation}
        \hat{N}(\lambda) = \frac{|B_{|d|}|}{(2\pi)^{|d|}} \vol(M) \lambda^{|d|} + O(\lambda^{|d|-1-\frac{n-1}{n+1}}).
    \end{equation}
    and 
    \begin{equation}
        \hat{N}(\lambda) -\hat{N}(\lambda- c/\lambda) = O (\lambda^{|d|-2}).
    \end{equation}
\end{lemma}
\begin{proof}
    By lemma \ref{zec}, we know that 
    \begin{equation}
        P_{Z_i}(m_i+\frac{\alpha_i}{4}) = C_i (m_i+\frac{\alpha_i}{4})^{d_i-1} + O ((m_i+\frac{\alpha_i}{4})^{d_i-3}). \quad for \quad m_i \ge M
    \end{equation}
    To use this polynomial expansion, we need to define 
    \begin{equation}
        \hat{N}_1(\lambda) : = \sum_{\substack{m \in \Z_{\geq M}^k \times \Z^{n-k} \\ |m+y| \le \lambda }} \prod_{i = 1}^k P_{Z_i} (m_i +\frac{\alpha_i}{4})
    \end{equation}
    First we need to show that the difference between $\hat{N}_1(\lambda)$ and $\hat{N}(\lambda)$ won't hurt us much and is of order $O(\lambda^{d-1})$.
    \begin{align*}
        \hat{N}(\lambda)-\hat{N}_1(\lambda) &= \sum_{\substack{m \in \Z_{\geq 0}^k \times \Z^{n-k} \\ |m+y| \le \lambda }} \prod_{i = 1}^k P_{Z_i} (m_i +\frac{\alpha_i}{4}) - \sum_{\substack{m \in \Z_{\geq M}^k \times \Z^{n-k} \\ |m+y| \le \lambda }} \prod_{i = 1}^k P_{Z_i} (m_i +\frac{\alpha_i}{4}) \\
        & = \sum_{\substack{m \in \bigcup_{i=1}^k(\{m_i \le M\}\bigcap\Z_{\geq 0}^k \times \Z^{n-k}) \\ |m+y| \le \lambda }} \prod_{i = 1}^k P_{Z_i} (m_i +\frac{\alpha_i}{4})
    \end{align*}
    Since the first k manifolds are all Zoll manifolds with dimension greater or equal than 2, if we only count the first M lattice points in one of those directions, then we will lose at least two dimension in our Weyl counting function, which is the following
    \[
    \hat{N}(\lambda)-\hat{N}_1(\lambda) = O(\lambda^{|d|-2})
    \]
    Now we rewrite the $\hat{N}_1$ as
    \[
    \hat{N}_1(\lambda) = \sum_{m \in \Z_{\geq M}^k \times \Z^{n-k} + y} \chi_D(|m|/\lambda) \prod_{i=1}^k P_{Z_i}(m_i)
    \]
    And we define the main part of $\hat{N}_1$ as $\hat{N}_2$
    \[
    \hat{N}_2(\lambda) = \sum_{m \in \Z_{\geq M}^k \times \Z^{n-k} + y} \chi_D(|m|/\lambda) \prod_{i=1}^k C_im_i^{d_i-1}
    \]
    Then we can use Lemma \ref{zec} to estimate the difference of $\hat{N}_1$ and $\hat{N}_2$, which is also of order $O(\lambda^{|d|-2})$
    \begin{align*}
    \hat{N}_1(\lambda)-\hat{N}_2(\lambda) &= \sum_{m \in \Z_{\geq M}^k \times \Z^{n-k} + y} \chi_D(|m|/\lambda) \prod_{i=1}^k \left( C_i m_i^{d_i-1} + O(m_i^{d_i-3})\right)- \hat{N}_1(\lambda)\\
    &= \sum_{m \in \Z_{\geq M}^k \times \Z^{n-k} + y} \chi_D(|m|/\lambda) \prod_{i=1}^k O(|m|^{|d|-n-2})\\
    &= O(\lambda^{|d|-2})
    \end{align*}
    Where the last step is just due to the fact that the number of lattice points inside the $\lambda$-radius ball $B_{\lambda}$ is just $O(\lambda^n)$.

    Now we need to get rid of the restriction $m_i \geq M$ such that it's easy to compute and use the proposition below. So we define $\hat{N}_3(\lambda)$ as
    \[
    \hat{N}_3(\lambda) = \sum_{m \in  \Z^{n} + y \bigcap \R_{+}^k \times \R^{n-k}} \chi_D(|m|/\lambda) \prod_{i=1}^k C_im_i^{d_i-1}
    \]
    The difference between $\hat{N}_2(\lambda)$ and $\hat{N}_3(\lambda)$ is just
    \[
    \hat{N}_3(\lambda) - \hat{N}_2(\lambda) = \sum_{m \in  (\Z^{n} + y) \bigcap \bigcup_{j=1}^k H_j} \chi_D(|m|/\lambda) \prod_{i=1}^k C_im_i^{d_i-1}
    \]
    where $H_j = \{ x \in \R_{+}^k \times \R^{n-k} : x_j \in [0,y_j+M)\}$, for $j = 1, \cdots, k$. Actually the difference of $\hat{N}_3(\lambda) $ and $\hat{N}_2(\lambda) $ are very similar to the difference of $\hat{N}_1(\lambda) $ and $\hat{N}(\lambda) $, we could estimate them in the same way. We will elaborate the estimate of the first one because it is easier.
    \[
    \prod_{i=1}^k C_i m_i^{d_i-1} = O(|m|^{|d|-n-d_j+1}) \quad x \in H_j
    \]
    Hence 
    \[
    \sum_{m \in  (\Z^{n} + y) \bigcap H_j} \chi_D(|m|/\lambda) \prod_{i=1}^k C_im_i^{d_i-1} = O(\lambda^{|d|-d_j})
    \]
    just by naive estimate that the number of lattice point $\sum_{m \in  (\Z^{n} + y) \bigcap H_j} \chi_D(|m|/\lambda)$ is of $O(\lambda^{n-1})$. Since all $d_j \geq 2$, $O(\lambda^{|d|-d_j})$ is better than $O(\lambda^{|d|-2})$, so we have 
    \begin{equation}
        \hat{N}_3(\lambda) - \hat{N}_2(\lambda) = O(\lambda^{|d|-2})
    \end{equation}
    Combine all the above reduction, we have 
    \begin{equation}
        \hat{N}(\lambda) - \hat{N}_3(\lambda) = O(\lambda^{|d|-2})
    \end{equation}
    If we could prove that $\hat{N}_3(\lambda)$ satisfy
    \begin{equation}
        \hat{N}_3(\lambda) = \frac{|B_{|d|}|}{(2\pi)^{|d|}} \vol(M) \lambda^{|d|} + O(\lambda^{|d|-1-\frac{n-1}{n+1}}).
    \end{equation}
    We can finish the proof. Now what we need is just to prove the following proposition.
\end{proof}

\begin{proposition}\label{Weightedlattice}
    Let $\R_{+}$ denote the nonnegative real numbers and let $y\in \R^n$. Consider a multi-index $d = (d_1,d_2,\cdots,d_k,1,1,\cdots,1) \in \N^n$ with $d_i \geq 2$ for $i = 1,\cdots,k$. Then we have a weighted lattice points counting estimates
    \begin{equation}\label{lattice1}
        \sum_{\substack{m \in (\Z^n + y) \cap \R^k_+ \times\R^{n-k} \\ |m| \leq \lambda}} m_1^{d_1-1} \cdots m_k^{d_k - 1} = \lambda^{|d|} \int_{B \cap \R^k_+ \times \R^{n-k}} x_1^{d_1-1} \cdots x_k^{d_k-1} \, dx + E(\lambda)
    \end{equation}
    where $E(\lambda) = O( \lambda^{|d|-1-\frac{n-1}{n+1}})$ uniformly in y. Also we have the following annulus estimates 
    \begin{equation}\label{lattice2}
        \sum_{\substack{m \in (\Z^n + y) \cap \R^k_+ \times\R^{n-k} \\ \lambda \leq |m| \leq \lambda +c/\lambda}} m_1^{d_1-1} \cdots m_k^{d_k - 1} = O(\lambda^{|d| - 2})
    \end{equation}
\end{proposition}

    For clarity, we set
    \[
    F(x) = \prod_{i=1}^k \chi_{[0,\infty)} x_i^{d_i-1}
    \]
    Then the equation \ref{lattice1} and \ref{lattice2} can be written as $I(\lambda),J(\lambda)$
    \[
    I(\lambda) = \sum_{m \in \Z^n + y } \chi_{B(\lambda)}(m)F(m) = \lambda^{|d|}\int_B F(x)dx +E(\lambda)
    \]
    and 
    \[
    J(\lambda) = \sum_{m \in \Z^n + y } \chi_{A(\lambda)}(m)F(m) = O(\lambda^{|d|-2})
    \]
    where $A(\lambda) = \{x: \lambda \leq |x| \leq \lambda +c/\lambda\}$ is just the small $\lambda$ annulus.

    Let $\rho$ be a smooth, nonnegative function supported in $B \subset \R^n$ with $\int_{\R^n}\rho(x)dx =1 $. For all $\epsilon > 0$, we set $\epsilon^{-n}\rho(\epsilon^{-1}x)$. It is not hard to see that $\rho_{\epsilon}$ is also a smooth function supported in the ball of radius $\epsilon$. And we have $\int_{\R^n}\rho_{\epsilon}(x)dx =1$. $\rho_{\epsilon}$ is our mollifier so we could define the mollified sum as the following
    \[
    I_{\epsilon} (\lambda) = \sum_{m \in \Z^n + y } \chi_{B(\lambda)}(m)*\rho_{\epsilon}(|m|)F(m) = \lambda^{|d|}\int_B F(x)dx +E_{\epsilon}(\lambda)
    \]
    We define $C_d = \int_B F(x)dx$. Note that 
    \[
    I_{\epsilon}(\lambda -\epsilon) \leq I(\lambda) \leq I_{\epsilon}(\lambda+\epsilon)
    \]
    Hence we have 
    \[
    E_{\epsilon}(\lambda-\epsilon) - C_d(\lambda^{|d|} - (\lambda -\epsilon)^{|d|} \leq E(\lambda) \leq  E_{\epsilon}(\lambda+\epsilon) + C_d((\lambda+\epsilon)^{|d|} - \lambda^{|d|}
    \]
    If we can prove that 
    \begin{equation}\label{Error}
        E_{\epsilon}(\lambda) \lesssim \epsilon^{-\frac{n-1}{2}} \lambda^{|d| - \frac{n+1}{2}} + \epsilon\lambda^{|d|-1}
    \end{equation}
    then it is easy to show that
    \begin{equation}
        E(\lambda) \lesssim \epsilon^{-\frac{n-1}{2}} \lambda^{|d| - \frac{n+1}{2}} + \epsilon\lambda^{|d|-1}
    \end{equation}
    We could optimize this formula by setting 
    \begin{equation}
        \epsilon = \lambda ^{-\frac{n-1}{n+1}}
    \end{equation}
    With this, we get $E(\lambda) \sim O(\lambda^{|d|-1-\frac{n-1}{n+1}})$, which help us finish the proof of the first part.

    Before using Poisson summation formula, we want to make some reductions. First, we could exchange the order of multiplication and convolution in the sum, which is 
    \[
    \tilde{I}_{\epsilon} (\lambda) = \sum_{m \in \Z^n + y } (\chi_{B(\lambda)}F)*\rho_{\epsilon}(m) = C_d\lambda^{|d|} +\tilde{E}_{\epsilon}(\lambda)
    \]
    we claim that
    \begin{equation}\label{Error2}
        |\tilde{E}_{\epsilon}(\lambda) - E_{\epsilon}(\lambda)|= |\tilde{I}_{\epsilon} (\lambda) -I_{\epsilon} (\lambda)| = O(\epsilon\lambda^{|d|-1})
    \end{equation}
    which will deduce \eqref{Error} if we can show 
    \begin{equation}
        \tilde{E}_{\epsilon}(\lambda) \lesssim \epsilon^{-\frac{n-1}{2}} \lambda^{|d| - \frac{n+1}{2}} + \epsilon\lambda^{|d|-1}
    \end{equation}

The following lemma will give us \eqref{Error2}

\begin{lemma}\label{commutator}
\[
|\tilde{I}_{\epsilon} (\lambda) -I_{\epsilon} (\lambda)| = O(\epsilon\lambda^{|d|-1})
\]
\end{lemma}
\begin{proof}
    F is Lipschitz continuous such that
    \[
        |F(x) - F(z)| \leq C(1 + |x|^{|d|-n-1}) |x - z| \qquad \text{if $|x - z| \leq 1$}
    \]
    \begin{align*}
        |\chi_{\lambda B} *\rho_\epsilon(x) F(x) - (\chi_{\lambda B} F) * \rho_\epsilon(x)| &= \left| \int_{\lambda B} \rho_\epsilon(x - z) (F(x) - F(z)) \, dz \right| \\
        &\leq \int_{\R^n} \rho_\epsilon(x - z) |F(x) - F(z)| \, dz \\
        &\leq C(1 + |x|^{|d|-n-1})\int_{\R^n} \rho_\epsilon(x - z) |x - z| \, dz\\
        &= C \epsilon (1 + |x|^{|d|-n-1}) \int_{\R^n} \rho(z) |z| \, dz.
    \end{align*}
    Moreover since both $(\chi_{\lambda B} * \rho_\epsilon) F$ and $(\chi_{\lambda B} F) *\rho_\epsilon$ are supported on $|x| \leq \lambda + \epsilon$,
    \[
        |I_\epsilon(\lambda) - \tilde I_\epsilon(\lambda)| \lesssim \epsilon \sum_{\substack{m \in \Z^n + y \\ |m| \leq \lambda + \epsilon}} (1 + |x|^{|d|-n-1}) = O(\epsilon \lambda^{|d|-1}).
    \]
\end{proof}

By the Poisson summation formula,
\begin{align*}
    \tilde E_\epsilon(\lambda) &= \sum_{m \in \Z^n + y} (\chi_{\lambda B} F)*\rho_\epsilon(m) - C_d \lambda^{|d|}\\
    &= \sum_{m \in \Z^n} e^{2\pi i\langle y,m \rangle} \widehat{\chi_{\lambda B} F}(m) \widehat \rho(\epsilon m) - C_d \lambda^{|d|}\\
    &= \sum_{m \in \Z^n \setminus 0} e^{2\pi i\langle y,m \rangle} \widehat{\chi_{\lambda B} F}(m) \widehat \rho(\epsilon m)\\
    &\hspace{8em} + \int_{\R^n} (\chi_{\lambda B} F)*\rho_\epsilon(x) \, dx - \int_{\R^n} \chi_{\lambda B}(x) F(x) \, dx.
\end{align*} 
The difference of the integrals on the last line is bounded by
\begin{equation}\label{poisson main term}
    \int_{\R^n} \int_{\R^n} \rho_\epsilon(x-z) |\chi_{\lambda B}(x) F(x) - \chi_{\lambda B}(z) F(z) | \, dz \, dx
\end{equation}
We cut the outer integral into two domains, $|x| \leq \lambda - \epsilon$ and $\lambda - \epsilon < |x| \leq \lambda$. The former contributes
\begin{multline*}
    \int_{|x| \leq \lambda - \epsilon} \int_{\R^n} \rho_\epsilon(x-z) |F(x) - F(z) | \, dz \, dx\\
    \lesssim \epsilon \int_{|x| \leq \lambda - \epsilon} (1 + |x|)^{|d|-n-1} \, dx = O(\epsilon \lambda^{|d|-1})
\end{multline*}
by the same argument as in the proof of Lemma \ref{commutator}. If $\lambda - \epsilon < |x| \leq \lambda$, the inner integral in \eqref{poisson main term} is $O(\lambda^{|d|-n})$. Hence, the latter contributes
\begin{multline*}
    \int_{\lambda - \epsilon < |x| \leq \lambda} \int_{\R^n} \rho_\epsilon(x-z) |\chi_{\lambda B}(x) F(x) - \chi_{\lambda B}(z) F(z) | \, dz \, dx\\
    \lesssim \int_{\lambda - \epsilon < |x| \leq \lambda} \lambda^{|d|-n} \, dx = O(\epsilon \lambda^{|d|-1}),
\end{multline*}
and so $\eqref{poisson main term}$ is $O(\epsilon \lambda^{|d|-1})$. We are finally left with
\[
    \tilde E_\epsilon(\lambda) = \sum_{m \in \Z^n \setminus 0} e^{2\pi i\langle y,m \rangle} \widehat{\chi_{\lambda B} F}(m) \widehat \rho(\epsilon m) + O(\epsilon \lambda^{|d|-1}).
\]

Since $F$ is homogeneous of degree $|d|-n$, we have
\[
    \widehat{\chi_{\lambda B} F}(m) = \lambda^{|d|} \widehat{\chi_B F}(\lambda m)
\]
Hence, the following proposition will complete our proof.

\begin{proposition} \label{poisson sum proposition}
With everything as above,
\[
    \lambda^{|d|} \sum_{m \in \Z^n \setminus 0} |\widehat{\chi_{B} F}(\lambda m)| |\widehat \rho(\epsilon m)| = O(\epsilon^{-\frac{n-1}{2}} \lambda^{|d| - \frac{n+1}{2}}).
\]
\end{proposition}

The Proposition finally yields
\[
    \tilde E_\epsilon(\lambda) = O(\epsilon \lambda^{|d|-1} + \epsilon^{-\frac{n-1}{2}} \lambda^{|d|-\frac{n+1}{2}}).
\]
As noted previously, we optimize by setting $\epsilon = \lambda^{-\frac{n-1}{n+1}}$.
The proposition will hinge on the estimates for $|\widehat{\chi_B F}|$, below.

\begin{lemma} \label{fourier estimate}
    Let $Q = \{ \xi \in \R^n : \sup_{i}|\xi_i| \leq 1 \}$. For all real $R \geq 1$,
    \[
        \lambda^{|d|} \sum_{m \in \Z^n \cap R Q \setminus 0} |\widehat{\chi_B F}(\lambda m)| = O(R^{\frac{n-1}{2}}\lambda^{|d|-\frac{n+1}{2}}).
    \]
\end{lemma}

Since $\widehat \rho$ is Schwartz, we bound it by $|\widehat \rho(\xi)| \leq C_N\min(1,|\xi|^{-N})$ for a suitably large $N$. Assuming the lemma, we use a diadic decomposition to write the sum in Proposition \ref{poisson sum proposition} as
\begin{align*}
    &\lambda^{|d|} \sum_{m \in Z^n \cap \frac{1}{\epsilon} Q \setminus 0} |\widehat{\chi_B F}(\lambda m)| |\widehat \rho(\epsilon m)| + \lambda^{|d|} \sum_{j = 0}^\infty \sum_{m \in Z^n \cap \frac{2^j}{\epsilon} (2Q \setminus Q)} |\widehat{\chi_B F}(\lambda m)||\widehat{\rho}(\epsilon m)| \\
    &\lesssim \lambda^{|d|} \sum_{m \in Z^n \cap \frac{1}{\epsilon} Q \setminus 0} |\widehat{\chi_B F}(\lambda m)| + \lambda^{|d|} \sum_{j=0}^\infty 2^{-Nj} \sum_{m \in Z^n \cap \frac{2^j}{\epsilon} (2Q \setminus Q)} |\widehat{\chi_B F}(\lambda m)|\\
    &\lesssim \epsilon^{-\frac{n-1}{2}} \lambda^{|d| - \frac{n+1}{2}}
\end{align*}
The proof of Lemma \ref{fourier estimate} is all that remains.

\begin{proof}
Let $\beta \in C_0^\infty(\R)$ with $\beta \equiv 1$ on $[-1,1]$ and $\supp \beta \subset [-2,2]$. We write
\[
    \tilde F(x) = F(x) \prod_{i = 1}^n \beta(x_i) = \prod_{i = 1}^n \beta(x_i) \begin{cases}
        x_i^{d_i - 1} \chi_{[0,\infty)}(x_i) & \text{if $i \leq k$,} \\
        1 & \text{ if $i > k$.}
    \end{cases}
\]
Integration by parts twice in each of the $x_i$ variables yields a bound
\begin{equation} \label{F tilde bound}
    |\widehat{\tilde F}(\xi)| \lesssim \prod_{i=1}^n \langle \xi_i \rangle^{-2}.
\end{equation}
Note $\chi_B F = \chi_B \tilde F$ and hence $\widehat{\chi_B F} = \widehat \chi_B * \widehat{\tilde F}$. Using the well-known fact that
\[
    |\widehat \chi_B(\xi)| \lesssim \langle \xi \rangle^{-\frac{n+1}{2}},
\]
we write
\begin{align}
    \nonumber \sum_{m \in \Z^n \cap R Q \setminus 0} |\widehat{\chi_B F}(\lambda m)| &\lesssim \sum_{m \in \Z^n \cap RQ \setminus 0} \int_{\R^n} \langle \lambda m - \eta\rangle^{-\frac{n+1}{2}} \prod_{i = 1}^n  \langle \eta_i \rangle^{-2} \, d\eta \\
    \nonumber &= \int_{2\lambda R Q} \prod_{i = 1}^n \langle \eta_i \rangle^{-2}  \sum_{m \in \Z^n \cap RQ \setminus 0}  \langle\lambda m - \eta\rangle^{-\frac{n+1}{2}} \, d\eta \\
    \label{two integrals} &\qquad + \int_{\R^n \setminus 2\lambda R Q} \prod_{i = 1}^n \langle \eta_i \rangle^{-2}  \sum_{m \in \Z^n \cap RQ \setminus 0}  \langle \lambda m - \eta\rangle^{-\frac{n+1}{2}} \, d\eta.
\end{align}
If $\eta \in 2\lambda RQ$, then by the integral test
\begin{align*}
    \sum_{m \in \Z^n \cap RQ \setminus 0} \langle \lambda m - \eta \rangle^{-\frac{n+1}{2}} &\lesssim \int_{RQ} |\lambda \xi - \eta|^{-\frac{n+1}{2}} \, d\xi \\
    &\leq \lambda^{-\frac{n+1}{2}}\int_{3RQ} |\xi|^{-\frac{n+1}{2}}\, d\xi \lesssim \lambda^{-\frac{n+1}{2}} R^{\frac{n-1}{2}}.
\end{align*}
Hence, the first integral in the last line of \eqref{two integrals} is bounded by
\[
    R^{\frac{n-1}{2}} \lambda^{-\frac{n+1}{2}} \int_{2\lambda RQ} \prod_{i = 1}^n \langle \eta_i \rangle^{-2} \, d\eta \lesssim  R^{\frac{n-1}{2}} \lambda^{-\frac{n+1}{2}}.
\]
On the other hand if $\eta \not\in 2\lambda R Q$ and $m \in RQ$, then $\langle\lambda m - \eta\rangle \approx |\eta|$ and so the second of the integrals in \eqref{two integrals} is bounded by
\begin{align*}
    &R^n \int_{\R^n \setminus 2\lambda R Q} |\eta|^{-\frac{n+1}{2}} \prod_{i = 1}^n \langle \eta_i \rangle^{-2} \, d\eta \\
    &\leq R^n \sum_{j = 1}^n \int_{\substack{|\eta_j| = \max_i |\eta_i| \\ |\eta_j| \geq 2\lambda R}} |\eta|^{-\frac{n+1}{2}} \prod_{i = 1}^n \langle \eta_i \rangle^{-2} \, d\eta \\
    &\leq R^n \sum_{j = 1}^n \int_{\substack{|\eta_j| = \max_i |\eta_i| \\ |\eta_j| \geq 2\lambda R}} |\eta_j|^{-\frac{n+5}{2}} \prod_{i \neq j} \langle \eta_i \rangle^{-2} \, d\eta \\
    &= R^n \sum_{j = 1}^n \int_{|\eta_j| \geq 2\lambda R} |\eta_j|^{-\frac{n+5}{2}} \left( \prod_{i \neq j} \int_{|\eta_i| \leq |\eta_j|} \langle \eta_i \rangle^{-2} \, d\eta_i \right) \, d\eta_j\\
    &\lesssim R^n \sum_{j = 1}^n \int_{|\eta_j| \geq 2\lambda R} |\eta_j|^{-\frac{n+5}{2}} \, d\eta_j \\
    &\lesssim R^{\frac{n-3}{2}} \lambda^{-\frac{n+3}{2}},
\end{align*}
which is better than the bound on the first integral. The lemma follows.
\end{proof}

\newpage


\bibliography{Thesis}
\bibliographystyle{abbrv}

%

\end{document}